\documentclass[reqno,12pt]{amsart}
\usepackage{amssymb,amsmath,amsthm,amsxtra,mathrsfs,calc, color}
\usepackage{enumerate}
\usepackage[margin=1in]{geometry}

\usepackage{mleftright}
\mleftright
\usepackage{nccmath}
\usepackage{cases}

\newtheorem{theorem}{Theorem}[section]
\newtheorem{lemma}[theorem]{Lemma}

\newtheorem{proposition}[theorem]{Proposition}

\theoremstyle{remark}

\numberwithin{equation}{section}

\newcommand{\pfrac}[2]{\left(\frac{#1}{#2}\right)}
\newcommand{\pMatrix}[4]{\left(\begin{matrix}#1 & #2 \\ #3 & #4\end{matrix}\right)}
\renewcommand{\pmatrix}[4]{\left(\begin{smallmatrix}#1 & #2 \\ #3 & #4\end{smallmatrix}\right)}

\renewcommand{\(}{\left(}
\renewcommand{\)}{\right)}

\DeclareMathAlphabet{\matheur}{U}{eur}{m}{n}

\newcommand{\Q}{\mathbb{Q}}

\newcommand{\Z}{\mathbb{Z}}

\DeclareMathOperator{\GL}{GL}

\newmuskip\pFqskip
\mathchardef\pFcomma=\mathcode`, 

\begin{document}

\title{A note on cusp forms as $p$-adic limits}

\author{Scott Ahlgren}
\address{Department of Mathematics, University of Illinois at Urbana-Champaign, Urbana,
IL 61801, USA} \email{sahlgren@illinois.edu}
\author{Detchat Samart}
\address{Department of Mathematics, University of Illinois at Urbana-Champaign, Urbana,
IL 61801, USA} \email{dsamart@illinois.edu}


\date{\today}
\thanks{The first author was supported by a grant from the Simons Foundation (\#208525 to Scott Ahlgren).}
\maketitle
\begin{abstract}
Several authors have recently proved results which express cusp forms as $p$-adic limits of
weakly holomorphic modular forms under repeated application of Atkin's $U$-operator.
The proofs involve techniques from the theory of weak harmonic Maass forms,
and in particular a result of Guerzhoy, Kent, and Ono 
on the $p$-adic coupling of mock modular forms and their shadows.
Here we obtain strengthened versions of these results using techniques from the theory of holomorphic modular forms.
\end{abstract}
\section{Introduction}

In a recent paper \cite{EO}, El-Guindy and Ono study a cusp form and a modular function related to the elliptic curve $y^2=x^3-x$.
Following their notation, define
\begin{align}
g(z)&=\eta^2(4z)\eta^2(8z)= \sum_{n\ge 1}a(n)q^n=q-2q^5-3q^9+6q^{13}+\cdots,\label{eq:gdef}\\
L(z)&=\frac{\eta^6(8z)}{\eta^2(4z)\eta^4(16z)}=\frac1q+2q^3-q^7-2q^{11}+\cdots,\label{eq:Ldef}\\
{F}(z) &= -g(z)L(2z)=\sum_{n\ge -1}C(n)q^n=-\frac{1}{q}+2 q^3+q^7-2 q^{11}+\cdots.\label{eq:Fdef}
\end{align}
The main result of \cite{EO} states that if  $p\equiv 3\pmod 4$ is a prime for which $p\nmid C(p)$, then as a $p$-adic limit, we have
\begin{equation}\label{eq:elg-ono}
\lim_{m\to\infty}\frac{ F\big | U(p^{2m+1})}{C(p^{2m+1})}=g.
\end{equation}
The proof involves the theory of harmonic Maass forms, and in particular a result of Guerzhoy, Kent, and Ono \cite{GKO}
on the $p$-adic coupling of mock modular forms and their shadows.  Similar results were proved in \cite{GKO} and \cite{BG}.

Our goal is to prove strengthened versions of these results. We use a direct method; it does not involve harmonic Maass forms
but rather an investigation of the action of the Hecke operators on a family of weakly holomorphic modular forms.
A similar approach was recently employed in  the study of the congruences of Honda and Kaneko \cite{AA:heckegrids}.
For the modular forms described above, we prove the following, of which \eqref{eq:elg-ono} is an immediate corollary. Note in addition that the $m=0$ case
of \eqref{eq:vpcp} gives
$p\nmid C(p)$.
Let $v_p(\cdot)$ denote the $p$-adic valuation on $\Z[\![q]\!]$.
\begin{theorem}\label{T:main}
Let $p\equiv 3 \pmod 4$ be  prime. Then for all integers $m\ge 0$ we have
\begin{align}
v_p(C(p^{2m+1})) &= m,\label{eq:vpcp}\\
v_p\left(\frac{{F}| U(p^{2m+1})}{C(p^{2m+1})}-g\right)&\ge m+1.\label{eq:thmmain}
\end{align}
\end{theorem}

In Theorems~\ref{T:main4} and \ref{T:main3} below we obtain similar improvements of results given in  \cite{GKO} and \cite{BG}.
It is clear that the present approach would give similar results for a number of other spaces of modular forms.

\section{Background}
If $k$ is an  integer, $f$ is a function of the upper half-plane, and
$\gamma=\pmatrix abcd \in \GL_2^+(\Q)$, we define
\[
	f (z)\big|_k \gamma := (\det \gamma)^{k/2} (cz+d)^{-k} f\pfrac{az+b}{cz+d}.
\]
If  $N\ge 1$, $k\in \Z$, and $\chi$ is a Dirichlet character modulo $N$,
let $M_k(N,\chi)$ be the space consisting of  functions $f$ which
satisfy  $f|_k\pmatrix abcd=\chi(d)f$ for all $\pmatrix abcd\in \Gamma_0(N)$ and which are
holomorphic on the upper half plane and at the cusps.  Let  $M_k^!(N,\chi)$
be the space of forms which are
meromorphic at the cusps, and   let $M_k^\infty (N,\chi)$ denote the subspace of   forms which  are holomorphic at all cusps of $\Gamma_0(N)$ other than $\infty$.  We  drop the character from this notation when it is trivial.  Each $f\in M_k^!(N,\chi)$ can be identified with its $q$-expansion; with $q:=\exp(2\pi iz)$ we have
$f(z)=\sum a(n)q^n$ for some coefficients $a(n)$.

For each positive integer $m$, the $U$ and  $V$-operators are defined on $q$-expansions by
\begin{align*}
\sum a(n)q^n \big | U(m) &:= \sum  a(mn)q^n,\\
\sum a(n)q^n \big | V(m) &:= \sum  a(n)q^{mn}.
\end{align*}
Let   $T_{k,\chi}(m)$ be the usual Hecke operator on $M_k^!(N,\chi)$.
If $p$ is prime, then
for $n\geq 1$ and $f\in M_k^!(N,\chi)$ we have
\begin{equation}\label{E:TUV}
f | T_{k,\chi}(p^n) = \sum_{j=0}^n \chi(p^j)p^{(k-1)j} f| U(p^{n-j}) | V(p^j).
\end{equation}
Define
\begin{equation*}
\Theta : =\frac1{2\pi i}\frac{d}{dz}=q \frac{d}{dq}.
\end{equation*}

\begin{lemma}\label{P:HD}
If $(m,N)=1$, then we have
\begin{equation}
T_{k,\chi}(m) : M_k^\infty (N,\chi) \rightarrow M_k^\infty (N,\chi).
\label{E:T}
\end{equation}
If $k\geq 2$ then
\begin{equation} \Theta^{k-1} : M_{2-k}^\infty (N,\chi) \rightarrow M_k^\infty (N,\chi). \label{E:Th}
\end{equation}
\end{lemma}
\begin{proof}
For the first statement,  it suffices to show that for each prime $p\nmid N$ we have
\begin{equation*}
T_{k,\chi}(p) : M_k^\infty (N,\chi) \rightarrow M_k^\infty (N,\chi).
\end{equation*}
We have
\begin{equation}\label{E:Hecke}
f|T_{k,\chi}(p) = p^{\frac k2-1}\left(\sum_{j=0}^{p-1} f\big |_k
\pMatrix 1j0p
 +  \chi(p) f\big |_k
 \pMatrix p001
\right).
\end{equation}
Let $r\in\mathbb{Q}$ be a cusp of $\Gamma_0(N)$ inequivalent to $\infty$ and choose $\gamma=
\pmatrix abcd
\in \operatorname{SL}_2(\mathbb{Z}) \backslash \Gamma_0(N)$ with $\gamma\infty = r$.
Given $j\in \{0,\dots,p-1\}$ set  $\lambda:=(a+c j, p)$.
By a standard argument
 (see e.g.  \cite[\S 6.2]{Iwaniec}) we find that
\[
\pMatrix 1j0p
\pMatrix abcd
=\pMatrix
  {\frac{a+cj}{\lambda}}{*}
  {\frac{cp}{\lambda}}{*}
\pMatrix   \lambda  *
  0 {\frac{p}\lambda}
 \]
 where the first matrix on the right is in  $\operatorname{SL}_2(\mathbb{Z}) \backslash \Gamma_0(N)$.
 It follows that each term from the sum on $j$ in  \eqref{E:Hecke} is holomorphic at cusps other than $\infty$.
To see that the last summand is also holomorphic at these cusps,  let  $\lambda':=(p, c)$.  Then
\[
\pMatrix p001
\pMatrix abcd
=\pMatrix
  {\frac{ap}{\lambda'}}*
  {\frac{c}{\lambda'}}*
\pMatrix
  {\lambda'}*0{\frac{p}{\lambda'}}
 \]
 where the first matrix on the right is in  $\operatorname{SL}_2(\mathbb{Z}) \backslash \Gamma_0(N)$.

Let $R_k$ be the Maass raising operator in weight $k$, so that we have the basic relation
\[R_{k-2}\(f \big|_{k-2}\gamma\)=\(R_{k-2} f\)\big|_{k}\gamma.\]
Bol's identity (see for example \cite[Lemma 2.1]{BOR}) states that for $k\geq 2$ we have
\[\Theta^{k-1}=\frac1{(-4\pi)^{k-1}}R_{k-2}\circ R_{k-4}\circ\cdots\circ R_{4-k}\circ R_{2-k}.\]
It follows that
\[\Theta^{k-1} : M_{2-k}^!(N,\chi) \rightarrow M_k^! (N,\chi)\]
and that
\[\(\Theta^{k-1}f\)\big|_k\gamma=\Theta^{k-1}\(f\big|_{2-k}\gamma\).
\]
The claim \eqref{E:Th} follows from these two facts.
\end{proof}

If $p\nmid 6N$ and $k\geq 0$, let $M_k^{(p)}(N)$ denote the subset of $M_k(N)$ consisting of forms whose coefficients are $p$-integral rational numbers.
If $f\in M_k^{(p)}(N)$, define the filtration
\[
w_p(f):=\inf\{k' : f\equiv g \pmod p \text{ for some }g \in M_{k'}^{(p)}(N)\}.
\]
We require two  facts, which can be found for example in \cite[\S1]{Joch}.
First, if $f\in M_k^{(p)}(N)$ and $w_p(f)\neq -\infty$, then $w_p(f)\equiv k\pmod{p-1}$.
Also, we have
\begin{equation}\label{E:fil1}
w_p\(f|V(p)\) = p \,w_p(f).
\end{equation}

\section{Proof of Theorem~\ref{T:main}}\label{sec:main}
  Recall  the definitions \eqref{eq:gdef}--\eqref{eq:Fdef},
and note that $F=F_1$ and $g=-F_{-1}$ in the notation of the next proposition.
\begin{proposition}\label{T:T2}
We have the following.
\begin{enumerate}
\item For every odd  integer $m\geq -1$  there exists a unique  $F_m\in M_2^{\infty}(32)\bigcap\Z[\![q]\!]$
of the form
\[F_m= -q^{-m}+O(q^3).\]
\item
Suppose that  $p$ is an odd prime and that $n\geq 0$. Then
\begin{equation*}
{F} | T_2(p^n) = p^n F_{p^n}+C(p^n)g.
\end{equation*}
\end{enumerate}
\end{proposition}
\begin{proof}

For each integer $r\ge 0$, let
\begin{equation*}
E_r(z) = -g(z)L^r(2z)=-\frac{\eta^2(4z)\eta^{6r}(16z)}{\eta^{2r-2}(8z)\eta^{4r}(32z)}= -q^{-2r+1}+2q^{-2r+5}+O(q^{-2r+9}).
\end{equation*}
Using standard criteria (see, e.g. \cite[Thm. 1.64, Thm. 1.65]{O:web}) we find that
 $E_r\in M_2^{\infty}(32)$.  The forms $F_m$ can then be constructed as linear combinations of forms $E_r$ with $2r-1\equiv m \pmod 4$.
 Uniqueness follows since the space $S_2(32)$ is one-dimensional.  This gives the first assertion.

From \eqref{E:TUV} we have
\[
F|T_2(p^n)=F|U(p^n)+\sum_{j=1}^{n-1}p^j F| U(p^{n-j}) | V(p^j)+p^nF|V(p^n).\]
Observe that
 \[
 {F}| U(p^n)=C(p^n)q + O(q^3)= C(p^n)g +O(q^3)
 \]
 and that
  \[
\sum_{j=1}^{n-1} p^j {F}| U(p^{n-j}) | V(p^j)+ p^n {F}| V(p^n)  = -p^n q^{-p^n} +O(q^3).
  \]
Assertion (2)  follows from assertion (1) together with Lemma~\ref{P:HD}.
\end{proof}

Before proving Theorem~\ref{T:main} we require two lemmas.
\begin{lemma}\label{L:Cp}
For each prime $p\equiv 3 \pmod 4$ and each integer $m\ge 0$ we have
\begin{equation*}\label{E:Cp}
C(p^{2m+1}) \equiv (-1)^m p^m \, C(p) \pmod {p^{m+1}}.
\end{equation*}
\end{lemma}
\begin{proof}
Lemma 2.3 and Corollary 2.4 of \cite{EO} show that for each $p\equiv 3\pmod 4$,  there is a modular function $\phi_p\in M_0^\infty(32)$
of the form
\begin{equation}\label{eq:phi_p}
\phi_p(z) = q^{-p} +C(p)q+O(q^3)
\end{equation}
 (we have corrected a sign error in the proof of the corollary).
 From Lemma~\ref{P:HD} we have
 \[
\Theta(\phi_p)  = -pq^{-p}+C(p)q+O(q^3)\in M_2^\infty(32).
\]
On the other hand,  Proposition~\ref{T:T2} gives
\[
{F}| T_2(p) = -pq^{-p}+C(p)q +O(q^3).
\]
Therefore
\begin{equation}\label{E:T2T}
{F}| T_2(p) = \Theta(\phi_p),
\end{equation}
or equivalently
\begin{equation}\label{E:Up1}
{F}| U(p) = \Theta(\phi_p)-p\,{F}| V(p).
\end{equation}
Applying $U(p^2)$ to both sides of \eqref{E:Up1} and arguing  inductively, we obtain the following for each $m\geq 0$:
\begin{equation}\label{E:Up2}
{F}| U(p^{2m+1}) = \sum_{k=0}^m (-1)^{m-k} p^{m-k} \Theta(\phi_p)| U(p^{2k})+(-1)^{m+1}p^{m+1}{F}| V(p).
\end{equation}

For any $k\ge 0$ we have $\Theta(\phi_p)| U(p^{2k}) \equiv 0 \pmod {p^{2k}}$.
Therefore for each each $m\ge 0$ we have
\begin{equation}\label{E:PT}
{F} | U(p^{2m+1}) \equiv (-1)^m p^m \Theta(\phi_p)  \pmod {p^{m+1}}.
\end{equation}
The lemma follows by
comparing  coefficients of $q$ in  \eqref{E:PT}.
\end{proof}

The authors of   \cite{EO} speculated that $v_p(C(p))=0$ for every prime $p\equiv 3 \pmod 4$.  We prove that this is the case.
\begin{lemma}\label{L:Cp2}
For each prime $p\equiv 3 \pmod 4$ we have $p\nmid C(p)$.
\end{lemma}
\begin{proof}
Assume to the contrary that $p \mid C(p).$ From  \eqref{E:T2T} and Proposition~\ref{T:T2} it follows that
\[
\Theta(\phi_p) = {F}| T_2(p) = pF_p +C(p)g \equiv 0 \pmod p,
\]
from which it follows that for some integral coefficients $A_p$ we have
\[
\phi_p \equiv q^{-p}+\sum_{n=1}^\infty A_p(np)q^{np} \pmod p.
\]
 Let
\[f(z)= \frac{\eta^8(32z)}{\eta^4(16z)} = q^8+4q^{24}+O(q^{40})\in M_2(32).\]
Then $f^p  \in M_{2p}(32)$ has the form
\[f^p \equiv \sum_{n= 8}^\infty B_p(np)q^{np}\equiv q^{8p}+\cdots\pmod p.\]
Since $\phi_p\in M_0^\infty(32)$, we find that
$h_p:=\phi_p f^p\in M_{2p}(32)$ has the form
\[h_p \equiv  \sum_{n =7}^\infty D_p(pn) q^{pn}\equiv q^{7p}+\cdots\pmod p.\]
so that
\begin{equation}\label{E:hp}
h_p \equiv h_p| U(p)|V(p) \pmod p.
\end{equation}
Using \eqref{E:fil1} we obtain
\[
w_p(h_p)=p\,w_p(h_p| U(p)).
\]
Since $w_p(h_p)\equiv 2p\pmod{p-1}$ and $p\mid w_p(h_p)$ we must have
$w_p(h_p)=2p$, so that  $w_p(h_p| U(p))=2$. Thus there exists $h_0\in M_2^{(p)}(32)$ such that
\[
h_0 \equiv h_p | U(p) = q^7+O(q^8) \pmod p.
\]
However, by examining a basis for the eight-dimensional space $M_2(32)$
we find that there is no  such form $h_0$. This provides the desired contradiction.
\end{proof}

\begin{proof}[Proof of Theorem~\ref{T:main}]
Assertion \eqref{eq:vpcp}  follows  from Lemmas~\ref{L:Cp} and \ref{L:Cp2}. To prove \eqref{eq:thmmain},
we use Proposition~\ref{T:T2} and \eqref{E:TUV} to find that
\begin{equation}\label{E:FU}
\frac{{F}| U(p^{2m+1})}{C(p^{2m+1})}- g =\frac{1}{C(p^{2m+1})}\left(p^{2m+1}F_{p^{2m+1}} - \sum_{j=1}^{2m+1}p^j {F} | U(p^{2m+1-j})| V(p^j)\right).
\end{equation}
Using  \eqref{E:TUV} we obtain
\[
{F}| T_2(p^{2m})
 = \sum_{j=1}^{2m+1}p^{j-1} {F} | U(p^{2m+1-j})| V(p^{j-1}).
\]
Since $C(n)=0$ for  $n\not\equiv 3 \pmod 4$, we see from Proposition~\ref{T:T2} that
${F}| T_2(p^{2m})=p^{2m}F_{p^{2m}}$.  It follows that
\begin{equation*}
\sum_{j=1}^{2m+1}p^{j} {F} | U(p^{2m+1-j})| V(p^{j}) = p^{2m+1}F_{p^{2m}}| V(p) \equiv 0 \pmod {p^{2m+1}}.
\end{equation*}
Assertion \eqref{eq:thmmain} now follows from
 \eqref{E:FU} and  \eqref{eq:vpcp}.
\end{proof}

\section{An example in weight  $4$ and level $9$}\label{sec:GKO}
In \cite{GKO}, the authors study the $p$-adic coupling of mock modular forms and their shadows.
As an application of their general result, they prove
two $p$-adic limit formulas involving the hypergeometric functions $_2F_1\left(\frac{1}{3},\frac{1}{3};1;z\right)$ and $_2F_1\left(\frac{1}{3},\frac{2}{3};1;z\right)$ evaluated at certain modular functions.
We will use the following notation:
\begin{align*}
g_1(z)&=\eta^8(3z)= \sum_{n\ge 1}a(n)q^n=q-8 q^4+20 q^7-70 q^{13}+\cdots\in S_4(9),\\
L_1(z)&=\frac{\eta^3(z)}{\eta^3(9z)}+3=\frac{1}{q}+5 q^2-7 q^5+3 q^8+15 q^{11}+\cdots,\\
G(z) &= g_1(z)L_1^2(z)=\sum_{n\ge -1}C(n)q^n=\frac{1}{q}+2 q^2-49 q^5+48 q^8+771 q^{11}+\cdots.
\end{align*}
After rewriting using (3.3) and (3.4) of \cite{GKO}, we find that each of the two formulas in Theorem~1.3 of \cite{GKO} is equivalent to the assertion that for every prime $p\equiv 2 \pmod 3$ with $p^3\nmid C(p)$ we have
\begin{equation}\label{E:g4}
\lim_{m\rightarrow \infty} \frac{G | U(p^{2m+1})}{C(p^{2m+1})}=g_1(z).
\end{equation}
Here we prove a strengthened version of this result.
\begin{theorem}\label{T:main4}
Let $p\equiv 2 \pmod 3$ be a prime. Then for each integer $m\ge 0$ we have
\begin{align}
v_p(C(p^{2m+1}))&=\begin{cases}
    3m+1, & \text{if } p=2,\\
    3m, & \text{if } p\neq2.
  \end{cases} \label{E:vp41}\\
v_p\left(\frac{G| U(p^{2m+1})}{C(p^{2m+1})}-g_1\right)&\ge \begin{cases}
    3m+2, & \text{if } p=2,\\
    3m+3, & \text{if } p\neq2.
  \end{cases}
 \label{E:vp42}
\end{align}
\end{theorem}
The proof follows the argument in Section \ref{sec:main}, so we give fewer details here.
\begin{proposition}\label{P:Fm4}
We have the following.
\begin{enumerate}
\item For every integer $m\ge -1$ with $3\nmid m,$ there exists a unique $G_m\in M_4^\infty(9)\bigcap\Z[\![q]\!]$ of the form \[G_m=q^{-m}+O(q^2).\]
\item Let $p\neq 3$ be  prime  and let $n$ be a nonnegative integer. Then we have
\begin{equation*}
G| T_4(p^n) = p^{3n}G_{p^n}+C(p^n)g_1.
\end{equation*}
\end{enumerate}
\end{proposition}

\begin{proof}
For each integer $r\ge 0,$ let
\[E_r(z) = g_1(z)L_1(z)^r = q^{1-r}+(5r-8)q^{4-r}+O(q^{7-r}).\]
Then $E_r(z)\in M_4^\infty(9)$. We construct each form $G_m$ by taking a linear combination of $E_r$ with $r-1\equiv m \pmod 3.$ Uniqueness
 follows since $S_4(9)$ is spanned by the form $g_1=G_{-1}$.

We  deduce assertion (2) as in the last section using \eqref{E:T}, \eqref{E:TUV}, and assertion (1).

\end{proof}
\begin{lemma}\label{L:Cp4}
If $p\equiv 2 \pmod 3$ is prime, then
\begin{equation*}
C(p^{2m+1}) \equiv (-1)^{m}p^{3m}C(p) \pmod {p^{3m+3}}.
\end{equation*}
\end{lemma}
\begin{proof}
Define
 \[
 \phi_2(z)= \frac{\eta^2(3z)}{\eta^6(9z)} = \sum_{n\ge -2}A_2(n)q^n=\frac{1}{q^2}-2 q-q^4+O\left(q^5\right).
 \]
It is seen from the  expression of $\phi_2$ as an infinite product that $A_2(n)=0$ if $n\not\equiv 1 \pmod 3$.
Similarly, if
\[
L_1(z) = \frac{\eta^3(z)}{\eta^3(9z)}+3 = \sum_{n\ge -1}b(n)q^n,
\]
then $b(n)=0$ for all $n\not\equiv 2 \pmod 3.$ Therefore, for each positive integer $l\equiv 2 \pmod 3$ there exist $c_0,c_1,\ldots, c_{\frac{l-2}{3}}\in\mathbb{Z}$ such that
\begin{equation*}
\phi_l(z) = \phi_2(z)\sum_{j=0}^{\frac{l-2}{3}}c_j L_1^{l-2-3j}(z) = q^{-l}+\sum_{n\ge 1} A_l(n)q^n \in M_{-2}^\infty(9),
\end{equation*}
with $A_l(n)\in\mathbb{Z}$ and $A_l(n)=0$ if $n\not\equiv 1 \pmod 3$
(these coincide with the forms $w_l$ in \cite[Prop. 3.1]{GKO}).
Since the constant term in the weight two modular form $\phi_l L_1$ must be zero, we find as in the last section that
$A_l(1)=-C(l)$. In particular, for any prime $p\equiv 2 \pmod 3$ we have
\[\phi_p = q^{-p}-C(p)q +O(q^2).\]
By Lemma \ref{P:HD}, we have \[\Theta^3(\phi_p)= -p^3q^{-p}-C(p)q+O(q^2)\in M_4^\infty(9).\] Hence it follows from Proposition \ref{P:Fm4} that
\begin{equation}\label{E:ph4}
\Theta^3(\phi_p) = -p^3 G_p -C(p)g_1 = -G| T_4(p) =-G| U(p) -p^3 G| V(p),
\end{equation}
so that
\begin{equation}\label{E:Up4}
G| U(p) = -\Theta^3(\phi_p)-p^3 G| V(p).
\end{equation}
Applying $U(p^2)$ iteratively  leads to
\begin{equation}\label{E:Up4m}
G| U(p^{2m+1}) = \sum_{l=0}^m (-1)^{m+1-l}p^{3(m-l)}\Theta^3(\phi_p)| U(p^{2l})+(-1)^{m+1}p^{3(m+1)}G| V(p)
\end{equation}
for any non-negative integer $m$.
Since $\Theta^3(\phi_p)| U(p^{2l}) \equiv 0 \pmod {p^{6l}}$, we have from \eqref{E:Up4m} that
\begin{equation}\label{E:Up4mc}
G| U(p^{2m+1}) \equiv (-1)^{m+1}p^{3m}\Theta^3(\phi_p) \pmod {p^{3m+3}}.
\end{equation}
Comparing  coefficients of $q$ in \eqref{E:Up4mc} gives the result.
\end{proof}

The authors of \cite{GKO}  verified that $p^3\nmid C(p)$ for every prime $p\equiv 2 \pmod 3$ less than $32,500$.
Here we prove
\begin{lemma}\label{L:Cp4d}
For every odd prime $p\equiv 2 \pmod 3$, we have $p\nmid C(p)$.
\end{lemma}
\begin{proof}
Suppose by way of contradiction that $p\equiv 2\pmod 3$ is an odd prime with $p\mid C(p).$ Then  \eqref{E:ph4} gives
\[\Theta^3(\phi_p) \equiv 0 \pmod p,\]
which implies that for some coefficients $A_p$ we have
\[
\phi_p\equiv q^{-p}+\sum_{n\ge 1}A_p(np)q^{np} \pmod p.
\]
Since $\phi_2$ has no zeros on the upper half plane (and does not vanish at any cusp), we have
$h_p:=\phi_p \phi_2^{-p}  \in  M_{2p-2}(9)$.
Moreover,
\[h_p\equiv  \sum_{n\ge p}D_p(pn)q^{pn}\equiv q^p+\cdots\pmod p.\]
Therefore
$h_p| U(p)|V(p) \equiv h_p \pmod p$
so that
$w_p(h_p)=p w_p(h_p| U(p))$.
Since $w_p(h_p)\equiv 2p-2\pmod {p-1}$ and $w_p(h_p)\equiv 0\pmod p$, we must have $w_p(h_p)=0$,
but this is impossible since $M_0(9)$ contains no non-constant elements.
\end{proof}

\begin{proof}[Proof of Theorem \ref{T:main4}]
Assertion \eqref{E:vp41} follows from Lemma~\ref{L:Cp4}, Lemma~\ref{L:Cp4d}, and the fact that $C(2)=2$.
Next, we use Proposition \ref{P:Fm4} and \eqref{E:TUV} to write
\begin{equation}\label{E:FU4}
\frac{G| U(p^{2m+1})}{C(p^{2m+1})}- g_1 =\frac{1}{C(p^{2m+1})}\left(p^{6m+3}G_{p^{2m+1}} - \sum_{j=1}^{2m+1}p^{3j} G | U(p^{2m+1-j})| V(p^j)\right).
\end{equation}
Since $C(n)=0$ for any $n\not\equiv 2 \pmod 3$, Proposition~\ref{P:Fm4} and  \eqref{E:TUV} give
\[\sum_{j=1}^{2m+1}p^{3j} G | U(p^{2m+1-j})| V(p^{j}) = p^3G| T_4(p^{2m})| V(p) = p^{6m+3}G_{p^{2m}}| V(p) \equiv 0 \pmod {p^{6m+3}}.\]
The result follows from  \eqref{E:FU4} and \eqref{E:vp41}.
\end{proof}

\section{An example in weight  $3$ and level $16$}\label{sec:BG}
In \cite{BG}  the authors establish an analogous representation of a weight $3$ cusp form
as a $p$-adic limit.  Let $\chi$ denote the non-trivial Dirichlet character modulo $4$, and define
\begin{align*}
g_2(z)&:= \eta^6(4z) = \sum_{n\ge 1}a(n)q^n=q-6 q^5+9 q^9+\cdots \in S_3(16,\chi),\\
L_2(z) &:= \frac{\eta^6(8z)}{\eta^2(4z)\eta^4(16z)}=\frac{1}{q}+2 q^3-q^7-2 q^{11}+\cdots,\\
H(z) &:= g_2(z)L_2^2(z)=\sum_{n\ge -1}C(n)q^n=\frac{1}{q}-2 q^3-13 q^7+26 q^{11}+\cdots.
\end{align*}
The two  formulas stated in the main theorem of \cite{BG} involve the hypergeometric function $_2F_1(\frac{1}{2},\frac{1}{2};1;z);$
after rewriting they are equivalent to the following statement: for every prime $p\equiv 3 \pmod 4$ with $p^2\nmid C(p)$ we have
\[\lim_{m\rightarrow \infty} \frac{H| U(p^{2m+1})}{C(p^{2m+1})}=g_2(z).\]
Here we prove
\begin{theorem}\label{T:main3}
For every prime $p\equiv 3 \pmod 4$ and every integer $m\ge 0$ we have
\begin{align}
v_p(C(p^{2m+1}))&=2m, \label{E:vp31}\\
v_p\left(\frac{H| U(p^{2m+1})}{C(p^{2m+1})}-g_2\right)&\ge 2m+2.\label{E:vp32}
\end{align}
\end{theorem}
We give only a sketch of the proof.
\begin{proposition}\label{L:Fm3} We have the following.
\begin{enumerate}
\item For every odd integer $m\ge -1,$ there exists a unique $H_m\in M_3^\infty(16,\chi)\bigcap\Z[\![q]\!]$ of the form
\[H_m = q^{-m}+O(q^3).\]
\item Let $p$ be an odd prime and let $n$ be a nonnegative integer. Then we have
\[ H | T_{3,\chi}(p^n) = \chi(p^n)p^{2n} H_{p^n}+C(p^n)g_2.\]
\end{enumerate}
\end{proposition}
\begin{proof}
For each integer $r\ge 0$ define
\[
E_r(z) := g_2(z)L_2^r(z)= \frac{\eta^{6r}(8z)}{\eta^{2r-6}(4z)\eta^{4r}(16z)}\in M_3^\infty(16,\chi).\]
We  construct the form $H_m$ with the desired properties by taking an appropriate linear combination of $E_r$, and uniqueness follows since $S_3(16,\chi)$ is one-dimensional.
Assertion (2) is proved as before.
\end{proof}

\begin{lemma}\label{L:Cp3}
If $p\equiv 3 \pmod 4$ is prime and $m\geq 0$  then
\[C(p^{2m+1}) \equiv p^{2m}C(p) \pmod {p^{2m+2}}.\]
\end{lemma}
\begin{proof}
For each  $l\ge 2$, let $\phi_l\in M_{-1}^\infty(16,\chi)$ be the form given in \cite[Lem. 3.3]{BG}. We have $\phi_2(z)=\frac{\eta^2(8z)}{\eta^4(16z)}$.
For  $l\ge 3$ we have
\[\phi_l(z) = \phi_2(z)P_l(L_2(z)), \]
where  $P_l(x)\in\mathbb{Z}[x]$ has $\deg P_l = l-2.$ Let $p\equiv 3 \pmod 4$ be  prime.  As above we find that
\[\phi_p(z)=q^{-p}-C(p)q+\sum_{n\ge 5}A_p(n)q^n.\]
It  follows from Proposition~\ref{P:HD} that
\begin{equation}\label{E:ph3}
\Theta^2(\phi_p) = p^2q^{-p}-C(p)q+O(q^5) \in M_3^\infty(16,\chi),
\end{equation}
and we deduce using Proposition \ref{L:Fm3} that
\[H| U(p) = H| T_{3,\chi}(p)+p^2 H| V(p) = -\Theta^2(\phi_p)+p^2 H| V(p).\]
Iteratively applying $U(p^2)$ results in
\begin{equation*}
H | U(p^{2m+1}) =-\sum_{l=0}^m p^{2(m-l)}\Theta^2(\phi_p)| U(p^{2l})+p^{2(m+1)}H| V(p),
\end{equation*}
so we have
\begin{equation}\label{E:Up3}
H | U(p^{2m+1}) \equiv -p^{2m}\Theta^2(\phi_p) \pmod {p^{2m+2}}.
\end{equation}
Comparing coefficients gives the result.
\end{proof}

\begin{lemma}\label{L:Cp3d}
For every prime $p\equiv 3 \pmod 4$ we have $p\nmid C(p).$
\end{lemma}
\begin{proof}
Suppose by way of contradiction that $p\mid C(p).$ Then  \eqref{E:ph3} and Lemma~\ref{L:Fm3} show that
$\Theta^2(\phi_p) \equiv 0 \pmod p,$
whence \[\phi_p \equiv q^{-p}+\sum_{n\ge 1}A_p(np)q^{np} \pmod p.\]
Let $f(z)=\frac{\eta^{12}(16z)}{\eta^6(8z)}= q^6+6q^{14}+O(q^{22})\in M_3(16,\chi).$
Then $h_p := \phi_p f^p  \in M_{3p-1}(16)$ has the form
\[h_p\equiv  \sum_{n \ge 5p}D_p(pn)q^{pn}\equiv q^{5p}+\cdots\pmod p,\]
so that
\[h_p \equiv h_p| U(p)|V(p) \pmod p.\]
Analyzing the filtration  yields $w_p(h_p) = 2p$ and  $w_p(h_p | U(p)) =2$.
However,  we find by examining a basis that there is no form $h_0\in M_2^{(p)}(16)$ with
$h_0\equiv q^5+\cdots\pmod p$.  This provides the desired contradiction.
\end{proof}
The proof of Theorem~\ref{T:main3} follows as before.

\bibliographystyle{plain}
\bibliography{cusp_padic_limits}
\end{document}